\documentclass{amsart}
\usepackage{graphicx}
\usepackage{latexsym}
\usepackage{hyperref}
\usepackage{amsfonts}
\usepackage[all]{xy}
\usepackage{amssymb, mathrsfs, amsfonts, amsmath}
\usepackage{amsbsy}
\usepackage{amsfonts}
\newtheorem*{acknowledgement}{Acknowledgement}
\newtheorem{corollary}{Corollary}

\newtheorem{lemma}{Lemma}

\newtheorem{theorem}{Theorem}

\numberwithin{equation}{section}

\title[four-manifold with positive curvature]{Four-manifolds with positive curvature}
\author{R. Di\'ogenes}
\author{E. Ribeiro Jr.}
\author{E. Rufino}

\address[R. Di\'ogenes]{UNILAB, Instituto de Ci\^{e}ncias Exatas e da Natureza, Campus dos Palmares, ROD. CE 060, Km 51, 62.785-000 - Acarape / CE, Brazil}\email{rafaeldiogenes@unilab.edu.br}

\address[E. Ribeiro Jr.]{Universidade Federal do Cear\'a - UFC, Departamento  de Matem\'atica, Campus do Pici, Av. Humberto Monte, Bloco 914,
60455-760, Fortaleza / CE, Brazil.}\email{ernani@mat.ufc.br}

\address[E. Rufino]{Universidade Federal de Roraima - UFRR, Departamento de Matem\'atica, Campus Paricarana, Av. Cap. Ene Garcez, 2413, 69310-000, Boa Vista / RR, Brazil} \email{elzimar.rufino@ufrr.br}

\thanks{E. Ribeiro Jr was partially supported by grants from CNPq/Brazil (Grant: 303091/2015-0), PRONEX-FUNCAP/CNPq/Brazil and CAPES/ Brazil - Finance Code 001}
\thanks{E. Rufino was partially supported by CAPES/Brazil}

\begin{document}

\begin{abstract}
In this note we prove  that a four-dimensional compact oriented half-confor\-mally flat Riemannian manifold $M^4$ is topologically $\mathbb{S}^{4}$ or $\mathbb{C}\mathbb{P}^{2},$ provided that the sectional curvatures all lie in the interval $[\frac{3\sqrt{3}-5}{4},\,1].$ In addition, we use the notion of biorthogonal (sectional) curvature to obtain a pinching condition which guarantees that a four-dimensional compact manifold is homeomorphic to a connected sum of copies of the complex projective plane or the $4$-sphere.
\end{abstract}

\maketitle

\section{Introduction}

A classical topic in Riemannian geometry is to study manifolds with positive sectional curvature. The sectional curvature is
the most natural generalization to higher dimensions of the Gaussian curvature of a surface, given that it controls the behavior of geodesics.  However, very few topological obstructions to positive sectional curvature are known, and many conjectures about this subject remain open, as for example the Hopf conjecture on $\Bbb{S}^{2}\times\Bbb{S}^{2},$ which is one of the oldest conjectures in global Riemannian geometry.

A compact (without boundary) Riemannian manifold $(M^{n},\,g)$ is said $\delta$-pinched if the sectional curvature $K$ satisfies 

\begin{equation}
1\geq K\geq\delta.
\end{equation} If the strict inequality holds, we say that $M^n$ is strictly $\delta$-pinched.

The notion of curvature pinching was introduced by Rauch \cite{Rauch} in 1951. In considering this notion Rauch was able to show that a compact simply connected Riemannian manifold which is strictly $(3/4)$-pinched  is a topological sphere. This curvature pinching was improved to $\delta=1/4$ in 1960 by Berger \cite{BergerASNSP} and Klingenberg \cite{Kl}. Such an improvement became known as the Topological Sphere Theorem. After almost 50 years Brendle and Schoen \cite{BS-JAMS} showed by an outstanding method that under the same curvature pinching of Berger and Klingenberg such a manifold must be diffeomorphic to the sphere, this result has been known as the Differentiable Sphere Theorem; see also \cite{BS}. Moreover, by combining the results of Berger \cite{Berger} and Petersen and Tao \cite{tao}, it is known that given a Riemannian manifold $M^{n}$ there is a real number $\varepsilon$ (unknown) such that if $M^n$  is $(\frac{1}{4}-\varepsilon)$-pinched, then $M^n$ is either homeomorphic to $\Bbb{S}^n$ or diffeomorphic to a spherical space form of rank $1.$ In \cite{Hulin}, Hulin used the classical Weitezenb\"ock formula to show that if a four-dimensional connected manifold $(M^4,\,g)$ is $(\frac{1}{4}-\gamma)$-pinched, for $\gamma<2,5\,.\,10^{-4},$ then the second Betti number of $M^4$ is less than or equal to $1.$ These results stimulated many interesting works. In the next subsection we quickly review some related results in order to draw the state of the art and put our results in perspective.

\subsection{Four-manifolds with positive sectional curvature}

In the last decades many mathematicians have been studied four-dimensional manifolds under suitable curvature pinching conditions. It is well known that four-manifolds display a peculiar features.  In large part this is attributed to the fact that on a four-dimensional oriented compact Riemannian manifold $(M^4,g),$ the bundle of 2-forms, denoted by $\Lambda^2M,$ can be invariantly decomposed as
$$\Lambda^2M=\Lambda^+M\oplus\Lambda^-M,$$ where $\Lambda^\pm M$ are the $\pm1$-eigenspaces of the Hodge star operator $\ast.$ Hence, the space of harmonic 2-forms $H^2(M^4;\Bbb{R})$ can be split as $H^2(M^4;\Bbb{R})=H^+(M^4;\Bbb{R})\oplus H^-(M^4;\Bbb{R}),$ where $H^\pm(M^4;\Bbb{R})$ stands for the space of positive and negative harmonic 2-forms, respectively. Furthermore,  the second Betti number $b_2$ of $M^4$ can be written as $b_2=b^{+}+b^{-},$ where $b^\pm={\rm dim}H^\pm(M^4;\Bbb{R}).$

For our purposes it is important to recall that a four-dimensional Riemannian manifold $M^4$ is said to be {\it positive definite} if and only if $b^{-}=0.$ In particular, when the signature of $M^4$ is non-zero we will always orient the manifold so as to make the signature positive.

It follows from Bourguignon  \cite{Bourguignon} and Ville \cite{Ville} that a $(\frac{4}{19}\approx0.2105)$-pinched four-dimen\-sio\-nal compact manifold is topologically the sphere $\Bbb{S}^4$ or the complex projective space $\Bbb{CP}^2.$ This pinching constant was improved in 1991 by Seaman \cite{SeamanPAMS} to $\approx0.1714.$   Recently, Di\'ogenes and Ribeiro \cite{DR} were able to show that a four-dimensional compact oriented connected Riemannian manifold which is $(\approx0.16139)$-pinched must be de\-fi\-nite. In particular, they showed that a four-dimensional compact oriented Einstein manifold $\frac{1}{10}$-pinched  is either topologically $\Bbb{S}^4$ or homothetically isometric to $\Bbb{CP}^2.$ Besides, the main result in \cite{Ribeiro} implies that a four-dimensional compact Einstein manifold $M^4$ with normalized Ricci curvature $Ric=1$ and sectional curvature $K\ge \frac{1}{12}$ must be isometric to either $\Bbb{S}^4$ or $\Bbb{CP}^2;$ see also \cite{XCao,Cui,Wu}. Indeed, it remains a challenging task to obtain new classification results under weaker curvature pinching conditions. 

Before stating our first result let us also recall that a metric on a four-dimensional manifold $M^4$ is {\it half conformallly flat} if it is sefdual or antiselfdual, namely, $W^{-}=0$ or $W^{+}=0,$ respectively, where $W$ stands for the Weyl tensor. Typical examples of half-conformally flat manifolds include the standard sphere, the complex projective plane or $K3$ surfaces with their Ricci-flat metrics. Other interesting
examples were built by LeBrun \cite{LeBrun}. For a nice overview on half-conformally flat manifolds see, for instance, [\cite{besse}, Chapter 13].

After these preliminary remarks we may announce our first result as follows.

\begin{theorem}\label{thhalf}
Let $(M^{4},g)$ be a four-dimensional compact oriented connected half-conformally flat Riemannian manifold whose sectional curvatures all lie in the interval $[\frac{3\sqrt{3}-5}{4},\,1].$ Then $M^4$ is topologically $\Bbb{S}^4$ or $\Bbb{CP}^2.$
\end{theorem}

An important observation comes from the fact that $\Bbb{CP}^2\sharp \Bbb{CP}^2$ admits metrics with $K\geq 0$ and moreover it is definite, but has $b_2 > 1.$ Indeed, it is very interesting to determine if $\Bbb{CP}^2\sharp \Bbb{CP}^2$ admits a metric with positive sectional curvature. It should be also emphasized that the lower bound for the sectional curvature considered in Theorem \ref{thhalf} (namely, $\approx0.049$) improves significantly the constant considered in Theorem 1 of \cite{DR} (namely, $\approx0.16139$). Besides, the conclusion of Theorem \ref{thhalf} is clearly an improvement.

As an attempt to better understand four-dimensional manifolds with positive sectional curvature, it is natural to investigate other curvature positivity conditions. In this perspective, we recall that, for each plane $P\subset T_{p}M$ at a point $p\in M^4,$   {\it the biorthogonal (sectional) curvature} of $P$ is defined by the following average of the sectional curvatures
\begin{equation}
\displaystyle{K^\perp(P)=\frac{K(P)+K(P^\perp)}{2}},
\end{equation} where $P^\perp$  is the orthogonal plane to  $P.$ Indeed, the sum of two sectional curvatures on two orthogonal planes plays a very crucial role on four-dimensional manifolds. This notion appeared previously in works by Singer and Thorpe \cite{ST}, Gray \cite{Gray}, Seaman \cite{Seam2}, Noronha \cite{NoronhaMC}, Costa and Ribeiro Jr \cite{CR}, Bettiol \cite{renato} and many others. The positivity of the biorthogonal curvature is an intermediate condition between positive sectional curvature and positive scalar curvature $s.$ Moreover, as it was observed by Singer and Thorpe \cite{ST} a four-dimensional Riemannian manifold $(M^4,g)$ is Einstein if and only if $K^\perp(P)=K(P)$ for any plan $P\subset T_pM$ at any point $p\in M^4.$  From Seaman \cite{SeamanTAMS} and Costa and Ribeiro \cite{CR}, $\Bbb{S}^4$ and $\Bbb{CP}^2$ are the only compact simply-connected four-dimensional manifolds with positive biorthogonal curvature that can have (weakly) $1/4$-pinched biorthogonal curvature, or nonnegative isotropic curvature, or satisfy $K^{\perp} \ge \frac{s}{24} > 0.$ In addition, by using this approach, Costa and Ribeiro \cite{CR} showed that the Yau's Pinching Conjecture is true in dimension $4.$ In \cite{renato2}, Bettiol proved that the positivity of biorthogonal curvature is preserved under connected sums. In particular, he showed that $\Bbb{S}^4,$ $\sharp^{m}\Bbb{CP}^2\sharp^{n}\overline{\Bbb{CP}}^2$ and $\sharp^{n}\big(\Bbb{S}^2\times \Bbb{S}^2\big)$ admit metrics with positive biorthogonal curvature. For more details see \cite{renato,renato2,CR,NoronhaMC} and \cite{Seam2}.

Now we may state our next result.

\begin{theorem}\label{thkperp}
Let $(M^{4},g)$ be a four-dimensional compact oriented connected Riemannian manifold satisfying $$K^{\perp}\geq \frac{s^{2}}{24(3\lambda_{1}+s)},$$ where $\lambda_1$ is the first eigenvalue of Laplacian operator and $s$ stands for the scalar curvature of $M^4.$ Then $M^{4}$ must be definite.
\end{theorem}

Since $\Bbb{S}^{2}\times \Bbb{S}^{2}$ is not definite,  Theorem \ref{thkperp} implies, in particular, that $\Bbb{S}^{2}\times \Bbb{S}^{2}$ does not admit a metric satisfying $$K^{\perp}\geq \frac{s^{2}}{24(3\lambda_{1}+s)}.$$  We also point out that this result was also independently observed by Cao and Tran (see Remark 1.1 and Theorem 1.1 (3) in \cite{CaoTran}). The methods designed for the proof of Theorem \ref{thkperp} was essentially inspired by \cite{gursky} (see also \cite{DR}).

In the sequel, as an application of Theorem \ref{thkperp} combined with results by Freedman \cite{Freedman} and Donaldson \cite{Donaldson} we get the following corollary.

\begin{corollary}
\label{cor1}
Let $(M^{4},g)$ be a four-dimensional simply connected compact oriented Riemannian manifold satisfying $$K^{\perp}\geq \frac{s^{2}}{24(3\lambda_{1}+s)}.$$ Then $M^{4}$ is homeomorphic to a connected sum $\Bbb{CP}^{2}\sharp\cdots\sharp\Bbb{CP}^{2}$ of $b_{2}$ copies of the complex projective plane (if $b_{2}>0$) or the $4$-sphere (if $b_{2}=0$).
\end{corollary}

It may be interesting to compare Corollary \ref{cor1} with Theorem 1.3 in \cite{Ribeiro}. In fact, our latter result requires the same pinching condition of Theorem 1.3 in \cite{Ribeiro}, however, it does not require conditions on the Weyl tensor and analyticity of the metric.

\section{Background}

Throughout this section we review some information and present lemmas that will be useful in the proof of the main results. We start recalling that on a four-dimensional oriented Riemannian manifold $M^4$ the bundle of $2$-forms can be invariantly decomposed as a direct sum $$\Lambda^2=\Lambda^{+}\oplus\Lambda^{-}.$$ In particular, the Weyl curvature tensor $W$ is an endomorphism of the bundle of 2-forms $\Lambda^2M=\Lambda^+M\oplus\Lambda^-M$ such that $$W=W^+\oplus W^-,$$ where $W^\pm:\Lambda^\pm M\longrightarrow\Lambda^\pm M$ are called of the {\it selfdual} and {\it antiselfdual} parts of $W.$ Thus, we may fix a point $p\in M^4$ and diagonalize $W^\pm$ such that $w_i^\pm,$ $1\le i \le 3,$ are their respective eigenvalues. In particular, they satisfy
\begin{equation}
\label{eigenvalues} w_1^{\pm}\leq w_2^{\pm}\leq w_3^{\pm}\,\,\,\,\hbox{and}\,\,\,\,w_1^{\pm}+w_2^{\pm}+w_3^{\pm}
= 0.
\end{equation} Next, as it was pointed out in \cite{CR} as well as \cite{Ribeiro}, the definition of biorthogonal curvature provides the following identities
\begin{equation}
\label{K1perp} K_1^\perp = \frac{w_1^+ + w_1^-}{2}+\frac{s}{12}
\end{equation} and
\begin{equation}\label{K3perp}
K_3^\perp = \frac{w_3^+ + w_3^-}{2}+\frac{s}{12},
\end{equation} where $K_1^\perp(p) = \textmd{min} \{K^\perp(P); P\subset T_{p}M \}$ and $K_3^\perp(p) = \textmd{max}\{K^\perp (P); P\subset T_{p}M \}.$ Furthermore, if $\mathcal{R}$ denotes the curvature of $M^4$ we get the following decomposition

\begin{equation}
\mathcal{R}=
\left(
  \begin{array}{c|c}
    \\
W^{+} +\frac{s}{12}Id & \mathring{Ric} \\ [0.4cm]\hline\\

    \mathring{Ric}^{\star} & W^{-}+\frac{s}{12}Id  \\[0.4cm]
  \end{array}
\right)=U+W^{+}+W^{-}+Z,
\end{equation} where $U=\frac{s}{12}Id_{\Lambda^{2}},$ $Z=\left(
  \begin{array}{cc}
0 & \mathring{Ric} \\

    \mathring{Ric}^{\star} & 0  \\
  \end{array}
\right)$ and $\mathring{Ric}:\Lambda^{-}\to \Lambda^{+}$ stands for the traceless part of the Ricci curvature of $M^4.$

Proceeding, we also remember that if $\omega$ is a $2$-form we have the following Weitzenb\"{o}ck formula
\begin{equation*}
\frac{1}{2}\Delta|\omega|^2=\langle\Delta\omega,\omega\rangle+|\nabla\omega|^2+\langle
\mathcal{N}(\omega),\omega\rangle,
\end{equation*}
where $\mathcal{N}$ is the Weitzenb\"{o}ck operator given by
\begin{eqnarray}
\label{WeitzenbockDef}
\langle \mathcal{N}(v_1\wedge v_2),w_1\wedge w_2\rangle&=&Ric(v_1,w_1)\langle v_2,w_2\rangle+Ric(v_2,w_2)\langle v_1,w_1\rangle\nonumber\\&&-Ric(v_1,w_2)\langle v_2,w_1\rangle -Ric(v_2,w_1)\langle v_1,w_2\rangle\nonumber\\&&+2\langle R(v_1,v_2)w_1,w_2\rangle.
\end{eqnarray} Here, $v_i$ and $w_i$ are tangent vectors; for more details see \cite{Ko} and \cite{SeamanPAMS}. We have adopted the opposite of the usual sign convention for the Laplacian, i.e., $\Delta f= {\rm div} (\nabla f).$

From now on we assume that $M^4$ is a four-dimensional $\delta$-pinched manifold, that is, the sectional curvature $K$ of $M^4$ satisfies 

\begin{equation}
1\geq K\geq\delta.
\end{equation} With this condition, as a slight modification of the proof of an useful inequality by Berger \cite{Berger} (see also \cite{DR}, Lemma 1) we obtain the following lemma.

\begin{lemma}
\label{lem1}
Let $(M^{4},\,g)$ be a $4$-dimensional oriented Riemannian manifold. Then we have:
\begin{equation*}
\langle \mathcal{N}(\omega),\omega\rangle \geq
4K_{1}^{\perp}|\omega|^{2}-\frac{1}{3}(s-12K_{1}^{\perp})\vert|\omega_{+}|^{2}-|\omega_{-}|^{2}\vert,
\end{equation*} where $\omega=\omega_++\omega_-$ and $\omega_\pm\in\Lambda^\pm M.$
\end{lemma}

\begin{proof}
First of all, given a point $p\in M$ there exists an oriented orthonormal basis for $T_p M$
$\{e_1,e_2,e_3,e_4\}$ satisfying $\ast(e_1\wedge e_2)=e_3\wedge e_4$ and such that
\begin{equation*}
\omega=\frac{\sqrt{2}}{2}(|\omega_+|+|\omega_-|)e_1\wedge e_2+\frac{\sqrt{2}}{2}(|\omega_+|-|\omega_-|)e_3\wedge
e_4
\end{equation*} at point $p.$ In view of the identity (\ref{WeitzenbockDef}) we get
\begin{eqnarray}\label{lem1eq1}
\langle \mathcal{N}(\omega),\omega\rangle&=&|\omega|^2(K_{13}+K_{14}+K_{23}+K_{24})-2R_{1234}(|\omega_+|^2-|\omega_-|^2)\nonumber\\
 &=&2|\omega|^2(K^\perp_{13}+K^\perp_{14})-2R_{1234}(|\omega_+|^2-|\omega_-|^2)\nonumber\\
 &\geq&4K^\perp_1|\omega|^2-2R_{1234}(|\omega_+|^2-|\omega_-|^2),
\end{eqnarray} where $K^\perp_{ij}$ stands for the biorthogonal curvature of plane $e_i\wedge e_j.$ Here, we used that $\langle \mathcal{N}(\omega_{+}),\omega_{-}\rangle=0.$ Next, we apply the Seaman's estimate $$|R_{ijkl}|\leq\frac{2}{3}(K^\perp_3-K^\perp_1)$$ in order to obtain

\begin{equation}\label{lem1eq2}
\langle
\mathcal{N}(\omega),\omega\rangle\geq4K^\perp_1|\omega|^2-\frac{4}{3}(K^\perp_3-K^\perp_1)(|\omega_+|^2-|\omega_-|^2).
\end{equation}

On the other hand, it follows from (\ref{eigenvalues}) that $$w_3^\pm\leq-2w^\pm_1.$$ This data combined
with (\ref{K1perp}) and (\ref{K3perp}) yields
\begin{eqnarray*}
K_3^\perp&=&\frac{w_3^++w_3^-}{2}+\frac{s}{12}\\
 &\leq&-(w_1^++w_1^-)+\frac{s}{12}\\
 &=&-2K_1^\perp+\frac{s}{6}+\frac{s}{12},
\end{eqnarray*} so that
\begin{equation}\label{lem1eq3}
K_3^\perp\leq\frac{s}{4}-2K_1^\perp.
\end{equation} Putting together (\ref{lem1eq3}) and (\ref{lem1eq2}) we infer
\begin{equation*}
\langle \mathcal{N}(\omega),\omega\rangle\geq4K^\perp_1|\omega|^2-\frac{1}{3}(s-12K^\perp_1)||\omega_+|^2-|\omega_-|^2|.
\end{equation*} This finishes the proof of the lemma.
\end{proof}

In order to introduce the next lemma, we need to fix notation. We consider the set given by $$G=\{x\wedge y \in \Lambda^{2};\,\,\,x\,\,\hbox{and}\,\,y\,\,\hbox{are unitary and orthogonal}\}.$$ Be\-si\-des, let $\mathcal{G}=G/\pm 1$ be the 2-dimensional Grassmannian manifold of $T_{p}M.$ From this, it follows that if $H\in \Lambda^{+}$ and $K\in \Lambda^{-},$ we have $\frac{H+K}{\sqrt{2}}\in \mathcal{G}$ if and only if $||
H||=||K||=1$ (cf. Lemma 1 in \cite{Ville2}). Now, we may state a result obtained by Ville (cf. Lemma 2 in \cite{Ville2}, see also \cite{Ville}), which plays an important role in this paper. 

\begin{lemma}[\cite{Ville2}]
\label{LemmaVi}
Let $M^4$ be a $4$-dimensional oriented $\delta$-pinched Riemannian manifold. Then:

\begin{enumerate}
\item For all $P\in \mathcal{G},$ we have $\delta \leq \left\langle \big(U+W\big)(P),P\right\rangle \leq 1.$
\item For all $H\in \Lambda^{+},$ we have $\delta\leq  u+\frac{1}{2}\left\langle W^{+}H,H\right\rangle\leq 1,$ where $u=\frac{s}{12}.$
\end{enumerate}

\end{lemma}

For what follows, since $W^{+}$ is a symmetric endomorphism of $\Lambda^{+},$ we may consider an orthonormal basis of $\Lambda^{+}$ given by $\{H_{1},\,H_{2},\,H_{3}\}$ such that $$W^{+}H_{i}=w_{i}^{+}H_{i},\,\,\hbox{for}\,\,i=1,\,2\,\,\hbox{or}\,\, 3.$$ Moreover, we consider $K_{i}=\dfrac{ \mathring{Ric}^{\star} (H_{i})}{||\mathring{Ric}^{\star} (H_{i})||}\in \Lambda^{-}$ and let $z_{i}=\langle\mathring{Ric}^\star(H_i),K_i\rangle.$ Thus, we know that it holds

$$||\mathring{Ric}^\star(H_i)||^ {2}=\langle z_{i}K_{i},\, z_{i}K_{i}\rangle=z_{i}^ {2}.$$ Hereafter, we set $\lambda_{i}^{-}=\langle W^{-}K_{i},\,K_{i}\rangle$ and $v_{i}=u+\frac{1}{2}w_{i}^ {+}.$ In particular, it follows from Lemma \ref{LemmaVi} that $\delta \leq v_{i}\leq 1.$ With these settings, we have the following lemma due to Ville (cf. Lemma 3 in \cite{Ville}, see also \cite{Ville2}).

\begin{lemma}[\cite{Ville}]
Let $M^4$ be a $4$-dimensional oriented $\delta$-pinched Riemannian manifold. Then:
\begin{equation}
||Z||^{2}\leq 2\displaystyle \sum_{i=1}^{3}A_{i}^{2},
\end{equation} where $||Z||^{2}=||\mathring{Ric}^{\star}||^{2}+||\mathring{Ric}||^{2}$ and $A_{i}=\min \{(1-v_{i}+\frac{1}{2}\lambda_{i}^{-}),\,(v_{i}+\frac{1}{2}\lambda_{i}^{-}-\delta)\}.$
\end{lemma} 

\begin{proof}
Since the proof of this lemma is very short, we include it here for sake of completeness. First of all,  we easily compute
\begin{eqnarray}
||Z||^{2}&=&||\mathring{Ric}^{\star}||^{2}+||\mathring{Ric}||^{2}=2||\mathring{Ric}^{\star}||^{2}\nonumber\\&=&2\sum_{i=1}^{3}||\mathring{Ric}^{\star}(H_{i})||^{2}\nonumber\\&=&2\sum_{i=1}^{3}\langle \mathring{Ric}^{\star}(H_{i}),\,K_{i}\rangle^{2}.
\end{eqnarray} Now, we need to estimate the value of $\langle \mathring{Ric}^{\star}(H_{i}),\,K_{i}\rangle.$ To do so, notice that by Lemma \ref{LemmaVi} we have
\begin{eqnarray}
\label{pl1}
\delta\leq \left\langle \mathcal{R}\Big(\frac{H_{i}\pm K_{i}}{\sqrt{2}}\Big), \, \frac{H_{i}\pm K_{i}}{\sqrt{2}}\right\rangle\leq 1.
\end{eqnarray} Moreover, one has

\begin{eqnarray}
\left\langle \mathcal{R}\Big(\frac{H_{i}\pm K_{i}}{\sqrt{2}}\Big), \, \frac{H_{i}\pm K_{i}}{\sqrt{2}}\right\rangle&=&u+\frac{1}{2}\langle W^{+}H_{i},\,H_{i}\rangle+\frac{1}{2}\langle W^{-}K_{i},\,K_{i}\rangle\nonumber\\&& +2\left\langle \mathring{Ric}\Big(\frac{\pm K_{i}}{\sqrt{2}}\Big),\,\frac{H_{i}}{\sqrt{2}}\right\rangle\nonumber\\&=&u+\frac{1}{2}w_{i}^{+}+\frac{1}{2}\lambda_{i}^{-}\pm \langle \mathring{Ric}(K_{i}),\,H_{i}\rangle \nonumber\\&=& v_{i}+\frac{1}{2}\lambda_{i}^{-}\pm \langle \mathring{Ric}^{\star}(H_{i}),\, K_{i}\rangle.
\end{eqnarray} This jointly with formula  (\ref{pl1}) gives

\begin{equation*}
\delta -v_{i}-\frac{1}{2}\lambda_{i}^{-}\leq \pm \langle \mathring{Ric}^{\star}(H_{i}),\,K_{i}\rangle\leq 1-v_{i}-\frac{1}{2}\lambda_{i}^{-},
\end{equation*}  from which it follows that

$$|\langle \mathring{Ric}^{\star}(H_{i}),\,K_{i}\rangle|\leq A_{i},$$ as wished.
\end{proof}

\vspace{1.30cm}

We are now in the position to present the proof of the main results.

\section{Proof of the main results}

\subsection{Proof of Theorem \ref{thhalf}}

\begin{proof}
To begin with, since $(M^4,g)$ is a compact half-conformally flat smooth manifold of positive scalar curvature we can use Proposition 2.4 of \cite{Noronha} to infer that $M^4$ is definite. Moreover, it is known by Synge's theorem that $M^4$ is simply connected. Thus, from Do\-naldson \cite{Donaldson} and Freedman \cite{Freedman}, $M^4$  is homeomorphic to a connected sum $\Bbb{CP}^{2}\sharp\cdots\sharp\Bbb{CP}^{2}$ of $b_{2}$ copies of the complex projective plane (if $b_{2}>0$) or the $4$-sphere (if $b_{2}=0$).
 
Now, it remains to prove that either $b_{2}=0$ or $b_{2}=1.$ To this end, it suffices to prove that 
\begin{equation}
|\tau(M)|<\frac{1}{2}\chi(M).
\end{equation} Indeed, without loss of generality we may assume that $\tau(M)>0.$ Therefore, by using the classical Gauss-Bonnet-Chern formula

$$\chi(M)=\frac{1}{8\pi^2}\int_M\left(\frac{s^2}{24}+|W^+|^2+|W^-|^2-\frac{1}{2}|\mathring{Ric}|^2\right)dV_g$$ jointly with the Hirzebrush's theorem $$\tau(M)=\frac{1}{12\pi^{2}}\int_M\left(|W^+|^2-|W^-|^2\right)dV_{g}$$ we arrive at
\begin{eqnarray*}
\chi(M)-2\tau(M)&=&\frac{1}{8\pi^2}\int_M\left(\frac{s^2}{24}-\frac{1}{3}|W^+|^2+\frac{7}{3}|W^-|^2-\frac{1}{2}|\mathring{Ric}|^2\right)dV_g.
\end{eqnarray*} 

Before proceeding, for simplicity, let us consider  $$\mathcal{F}(g)=\left(\frac{s^2}{24}-\frac{1}{3}|W^+|^2+\frac{7}{3}|W^-|^2-\frac{1}{2}|\mathring{Ric}|^2\right)$$ and hence, we easily see that

\begin{equation}
\label{kl}\chi(M)-2\tau(M)=\frac{1}{8\pi^2}\int_{M}\mathcal{F}(g)dV_g.
\end{equation} So, we need to show the positivity of the right hand side of (\ref{kl}). But, in view of the pinching condition $\delta\leq K\leq1$ we may use  Eq. (1) in \cite{Ville} (see also Lemmas 1.3 and 1.4  in \cite{Ville}) to obtain

\begin{eqnarray}
\label{deg}
\mathcal{F}(g)&\geq&\frac{10}{9}\left(\sum_{i=1}^3v_i\right)^2-\frac{4}{3}\sum_{i=1}^3v_i^2+\frac{7}{2}\alpha^{2}\nonumber\\&&-2\sum_{i=1}^3 \min\{\big(1-v_{i}-\frac{\lambda_{i}^{-}}{2}\big)^ {2},\big(v_{i}+\frac{\lambda_{i}^{-}}{2}-\delta\big)^{2}\},
\end{eqnarray} where $\alpha=\max |\lambda_{i}^{-}|.$ For more details see \cite{Ville}.

From now on, since $(M^4,g)$ is half-conformally flat, we assume that $W^{-}=0.$ Whence, it follows that $\lambda_{i}^{-}=0$ for $i=1,\,2$ or $3$ and consequently, $\alpha=0.$ Thus, by (\ref{deg}) one obtains

\begin{equation}
\frac{\mathcal{F}(g)}{2}\geq\frac{5}{9}\left(\sum_{i=1}^3v_i\right)^2-\frac{2}{3}\sum_{i=1}^3v_i^2-\sum_{i=1}^3m(v_i)^2,\end{equation} where $m(x)=\min\{1-x,x-\delta\}.$

In order to proceed, we introduce the function $f:\Bbb{R}^3\to\Bbb{R}$ given by $$f(x_1,x_2,x_3)=\frac{5}{9}\left(\sum_{i=1}^3x_i\right)^2-\frac{2}{3}\sum_{i=1}^3x_i^2-\sum_{i=1}^3m(x_i)^2$$ and moreover, let us consider the set $$E=\{(x_1,x_2,x_3)\in\Bbb{R}^3;\delta\leq x_1\leq x_2\leq x_3\leq1\}.$$ Then, easily one verifies that

$$Hess\,f=\frac{10}{9}\left(\begin{array}{ccc}
-2 & 1 & 1 \\
1 & -2 & 1 \\
1& 1 & -2
\end{array}\right),$$ whose the eigenvalues are $\{-3,-3,0\}.$ Therefore, $f$ is concave in $E$ and then it suffices to show the positivity of $f$ in the following points: $(\delta,\delta,\delta),$ $(\delta,\delta,1),$ $(\delta,1,1)$ and $(1,1,1).$ In fact, it is not hard to check that 

\begin{itemize}
\item $f(\delta,\delta,\delta)=3\delta^2;$
\item $f(\delta,\delta,1)=\frac{8\delta^2+20\delta-1}{9};$
\item $f(\delta,1,1)=\frac{-\delta^2+20\delta+8}{9};$
\item $f(1,1,1)=3.$
\end{itemize} Whence, taking into account that $\delta\geq\frac{3\sqrt{3}-5}{4},$ we conclude that $f$ is nonnegative, and we therefore have  $$\frac{\mathcal{F}(g)}{2}\geq f(v_1,v_2,v_3)\geq0.$$ Next, since there exist points where $f$ is positive on $E$ we deduce $$\int_M \mathcal{F}(g)dV_g>0$$ and hence, this data into (\ref{kl}) yields  $|\tau(M)|<\frac{1}{2}\chi(M).$ Finally, it is straightforward to check that $b_2=0$ or $b_2=1.$ Thereby, according to Freedman \cite{Freedman} (see also Donaldson \cite{Donaldson}) $M^4$ is homeomorphic to the complex projective space $\Bbb{CP}^2$ or the $4$-sphere $\Bbb{S}^4.$

So, the proof is completed.
\end{proof}

\subsection{Proof of Theorem \ref{thkperp}}

\begin{proof}
The first part of the proof follow the ideas outlined in \cite{Ribeiro}, which was partially inspired by \cite{gursky}. For the sake of completeness we include here all details. Indeed, we argue by contradiction, assuming that $M^4$ is indefinite. In this case, there exist nonzero harmonics 2-forms $\omega_+$ and $\omega_-$ such that

$$\int_M\Big(\left(|\omega_+|^2+\varepsilon\right)^{\frac{1}{4}}-t\left(|\omega_-|^2+\varepsilon\right)^{\frac{1}{4}}\Big)dV_g=0,$$
for any $\varepsilon>0$ and $t=t(\varepsilon).$ Next, following the same steps of the first part of the proof of Theorem 1 in \cite{DR} (see Eq. (3.6) in \cite{DR}) we achieve at

\begin{eqnarray}
\label{th1eq5}
0&\geq&\lambda_1\int_M\left(|\omega_+|^{\frac{1}{2}}-t|\omega_-|^{\frac{1}{2}}\right)^2 dV_g\nonumber\\&&+\int_M\left(|\omega_+|^{-1}\langle
 \mathcal{N}(\omega_+),\omega_+\rangle+t^2|\omega_-|^{-1}\langle
  \mathcal{N}(\omega_-),\omega_-\rangle\right)dV_g.
\end{eqnarray}

Proceeding in analogy with \cite{DR} we choose $X_+=|\omega_+|^{-\frac{1}{2}}\omega_+$ and $X_-=t|\omega_-|^{-\frac{1}{2}}\omega_-.$ Hence, easily one verifies that
$X_\pm\in\Lambda^\pm,$ $|X_+|=|\omega_+|^{\frac{1}{2}}$ and $|X_-|=t|\omega_-|^{\frac{1}{2}}.$ Therefore, by considering $X=X_{+}+X_{-},$ it follows from (\ref{th1eq5}) that
\begin{equation*}
0\geq\int_M\left\{\lambda_1\left(|X_+|-|X_-|\right)^2+\langle \mathcal{N}(X),X\rangle\right\}dV_g
\end{equation*} and then by Lemma \ref{lem1} we deduce
\begin{equation*}
0\geq\int_M\left\{\lambda_1\left(|X_+|-|X_-|\right)^2+4K_{1}^{\perp}|X|^2-\frac{1}{3}(s-12K_1^{\perp})\left||X_+|^2-|X_-|^2\right|\right\}dV_g.
\end{equation*} Since $|X_+|=|\omega_+|^\frac{1}{2}$ and $|X_-|=t|\omega_-|^\frac{1}{2},$ the above expression can be written succinctly as
\begin{eqnarray}
\label{th1eq6}0&\geq&\int_M\Big(\lambda_1(|\omega_+|-2t|\omega_+|^\frac{1}{2}|\omega_-|^\frac{1}{2}+t^{2}|\omega_-|)+4K_{1}^{\perp}|\omega_+|\nonumber\\
 &&+4K_{1}^{\perp}t^{2}|\omega_-|-\frac{1}{3}(s-12K_{1}^{\perp})\vert|\omega_+|-t^{2}|\omega_-|\vert\Big)dV_{g}.
\end{eqnarray} Notice that the integrand of (\ref{th1eq6}) is a quadratic function of $t.$ Thus, for simplicity, we may set
\begin{eqnarray*}
\mathcal{P}(t)&=&\lambda_1\big(|\omega_+|-2t|\omega_+|^\frac{1}{2}|\omega_-|^\frac{1}{2}+t^{2}|\omega_-|\big)+4K_{1}^{\perp}|\omega_+|\nonumber\\&&+4K_{1}^{\perp}t^{2}|\omega_-|-\frac{1}{3}(s-12K_{1}^{\perp})\vert|\omega_+|-t^{2}|\omega_-|\vert.
\end{eqnarray*}

Before proceeding, given a point $p\in M^4,$ let us consider the sets $$A=\{t;\,\,|\omega_+|\geq t^2|\omega_-| \mbox{ at }p\}$$ and $$B=\{t;\,\,|\omega_+|<t^2|\omega_-| \mbox{ at }p\}.$$ Therefore, the definition of $\mathcal{P}(t)$ implies that 
\begin{eqnarray}\label{th1eq7}
\mathcal{P}(t)&=&[\lambda_1+4K_{1}^{\perp}-\frac{1}{3}(s-12K_{1}^{\perp})]|\omega_+|-2\lambda_{1}|\omega_+|^\frac{1}{2}|\omega_-|^\frac{1}{2}t\nonumber\\&&+[\lambda_1+4K_{1}^{\perp}+\frac{1}{3}(s-12K_{1}^{\perp})]|\omega_-|t^2,
\end{eqnarray} in $A$ and
\begin{eqnarray}\label{th1eq8}
\mathcal{P}(t)&=&[\lambda_1+4K_{1}^{\perp}+\frac{1}{3}(s-12K_{1}^{\perp})]|\omega_+|-2|\omega_+|^\frac{1}{2}|\omega_-|^\frac{1}{2}t\nonumber\\&&+[\lambda_1+4K_{1}^{\perp}-\frac{1}{3}(s-12K_{1}^{\perp})]|\omega_-|t^2,
\end{eqnarray} in $B.$ In both cases, the discriminant $\Delta$ of $\mathcal{P}(t)$ is given by
\begin{eqnarray*}
\Delta&=&4\lambda_{1}^{2}|\omega_+||\omega_-|-4[\lambda_{1}+4K_{1}^{\perp}-\frac{1}{3}(s-12K_{1}^{\perp})][\lambda_{1}+4K_{1}^{\perp} +\frac{1}{3}(s-12K_{1}^{\perp})]|\omega_+||\omega_-|\\
 &=&4|\omega_+||\omega_-|\{\lambda_{1}^{2}-[(\lambda_{1}+4K_{1}^{\perp})^{2}-\frac{1}{9}(s-12K_{1}^{\perp})^{2}]\}\\
 &=&4|\omega_+||\omega_-|[\lambda_{1}^{2}-(\lambda_{1}^{2}+8\lambda_{1}K_{1}^{\perp}+16(K_{1}^{\perp})^{2})+\frac{1}{9}(s^{2}-24K_{1}^{\perp}s +144(K_{1}^{\perp})^{2})]\\
 &=&\frac{4}{9}|\omega_+||\omega_-|(-72\lambda_{1}K_{1}^{\perp}+s^{2}-24K_{1}^{\perp}s).
 \end{eqnarray*} Hence, the condition $K^\perp\geq\frac{s^2}{24(3\lambda_1+s)}$ guarantees that $\Delta\leq 0.$

 On the other hand, it follows from (\ref{th1eq7}) that $$[\lambda_1+4K_{1}^{\perp}+\frac{1}{3}(s-12K_{1}^{\perp})]|\omega_-|\geq 0,$$ where we used that $s-12K_1^\perp=-6(w_1^++w_1^-)\geq0.$ Otherwise, if $\mathcal{P}(t)$ is given by (\ref{th1eq8}), it follows from our assumption that it holds
\begin{eqnarray*}
\Big(\lambda_1+4K_{1}^{\perp}-\frac{1}{3}(s-12K_{1}^{\perp})\Big)|\omega_-|&\geq&\Big(\lambda_1+\frac{s^2}{6(3\lambda_1+s)}-\frac{1}{3}\big(s-\frac{s^2}{2(3\lambda_1+s)}\big)\Big)|\omega_-|\\
 &=&\frac{6\lambda_1(3\lambda_1+s)+s^2-2(3\lambda_1+s)s+s^2}{6(3\lambda_1+s)}|\omega_-|\\
 &=&\frac{18\lambda_1^2+6\lambda_1s+2s^2-6\lambda_1s-2s^2}{6(3\lambda_1+s)}|\omega_-|\\
 &=&\frac{18\lambda_1^2}{6(3\lambda_1+s)}|\omega_-|\\
 &\geq&0.
\end{eqnarray*} Thereby, in both cases we have $\mathcal{P}(t)\geq 0.$ Next, since $p$ is an arbitrary point we conclude from (\ref{th1eq6}) that $\mathcal{P}(t)\equiv0.$ Now, it suffices to use (\ref{th1eq7}) and (\ref{th1eq8}) to deduce $$[\lambda_1+4K_{1}^{\perp}\pm\frac{1}{3}(s-12K_{1}^{\perp})]|\omega_-|=0.$$ But, taking into account that $\lambda_1+4K_{1}^{\perp}\pm\frac{1}{3}(s-12K_{1}^{\perp})>0$ we conclude $|\omega_-|=0.$ This yields the desired contradiction, and therefore forces the intersection form of $M^4$ to be definite. To conclude, we invoke again Freedman \cite{Freedman} and Donaldson \cite{Donaldson} to deduce that $M^4$ is homeomorphic to the complex projective space $\Bbb{CP}^2$ or the $4$-sphere $\Bbb{S}^4.$ 

This finishes the proof of the theorem.
\end{proof}

\begin{acknowledgement}
The authors want to thank the referee for his careful reading, relevant remarks and valuable suggestions. Moreover, the authors want to thank Xiaodong Cao, Hung Tran, Ezio Costa and Qing Cui for their valuable comments and helpful conversations about this subject. In particular, we would like to thank Xiadong Cao and Qing Cui for pointing out the preprints \cite{CaoTran} and \cite{Cui}. 
\end{acknowledgement}

\end{document}